\documentclass[reqno]{amsart}

\usepackage{amsmath}
\usepackage{amssymb}
\usepackage{amsfonts}
\usepackage{graphicx}
\usepackage{amsthm}
\usepackage{enumerate}
\usepackage{lscape}
\usepackage{dsfont}
\usepackage{color}
\usepackage{mathtools}

\usepackage{setspace}
\newtheorem{theor}{Theorem}[section]
\newcounter{tmp}

\newtheorem{lem}[theor]{Lemma}
\newtheorem*{guess}{Conjecture}

\newcommand{\CP}{\mathds{C}\mathrm{P}}

\newcommand{\f}{\rightarrow}                  
\newcommand{\de}{\partial}          
\newcommand{\dd}{\mathrm{D}}

\newcommand{\R}{\mathbb{R}}

\newcommand{\C}{\mathbb{C}}

\newcommand{\K }{K\"{a}hler}

\begin{document}
\title[Projectively induced rotation invariant K\"{a}hler metrics]{Projectively induced rotation invariant \mbox{K\"{a}hler metrics}}

\author{Filippo Salis}
\address{Dipartimento di Matematica e Informatica, Universit\`a di Cagliari\\Via Ospedale 72, 09124 Cagliari (Italy)}
\email{filippo.salis@gmail.com}

\thanks{}

\begin{abstract}
We classify \K -Einstein manifolds admitting a \K\ immersion into a finite dimensional complex projective space endowed with the Fubini--Study metric, whose codimention is less or equal than 3 and whose metric is rotation invariant.
\end{abstract}
 
\maketitle
\tableofcontents

\section{Introduction}
In the present paper we address the problem of  studying holomorphic and isometric (i.e. K\"ahler) immersions of K\"ahler--Einstein (from now on KE) manifolds
 into the finite dimensional complex projective space $\CP^N$ endowed with the Fubini--Study metric $g_{FS}$.
In particular, we believe the validity of the following
\begin{guess}\label{conj1}
Every \K -Einstein manifold which admits a \K\ immersion into $(\CP^N,g_{FS})$, with $N<\infty$, is an open subset of a compact homogeneous space.
\end{guess}
An explicit example of non-homogeneous KE manifolds which admit a \K\ immersion into $(\CP^\infty,g_{FS})$ can be found in \cite{wallart}.\\

The main result of the paper consists in proving the above mentioned conjecture in the case of \emph{rotation invariant} \K\ metrics when the codimension with respect to the target projective space is $\leq3$. This is shown in the next section via the following two theorems.
\begin{theor}\label{mainteor1}
The Einstein constant $\lambda$  of a KE rotational invariant and projectively induced $n$-dimensional manifold $M^n$ is a positive rational number less or equal to $2(n+1)$. Hence, if $M^n$ is complete then $M^n$ is compact and simply connected.
\end{theor}
\begin{theor}\label{cod3}
Let $(M,g)$ be an $n$-dimensional KE manifold whose metric is rotation invariant. Then $(M,g)$ admits a K\"ahler immersion into $\CP^{n+k}$ for $k$ less than or equal to $3$, if and only if $(M,g)$ is an open subset of $(\CP^n,g_{FS})$, $(\CP^2,2g_{FS})$ or \mbox{$(\CP^1\times\CP^1,g_{FS}\oplus g_{FS})$}.
\end{theor}
Throughout the paper a \K\ metric $g$ on a complex manifold $M$  is \emph{rotation invariant} if  there exists a point $p\in M$, a coordinate system $(z_1,\dots, z_n)$ centered in $p$ and a K\"ahler potential $\Phi$ for $g$ on a neighborhood of $p$ such that $\Phi$ is rotation invariant in $(z_1,\dots, z_n)$, i.e. it  only depends on $|z_1|^2,\dots,|z_n|^2$.
Furthermore, a K\"ahler metric $g$ on a connected complex manifold $M$ will be called \emph{projectively induced}  if there exist a point $p\in M$,  a neighbourhood $V$ of $p$ and a K\"ahler immersion $f\!:V\f\CP^N$, with $N<\infty$, i.e. $f^*g_{FS}=g_{|V}$. If $M$ is projectively induced at $p$, then it is at any other point (cfr. \cite[Theorem 10]{Cal}). Moreover,  in the case whether $M$ is also simply connected then the \K\ immersion $f\!:V\f\CP^N$ can be extended to a \K\ immersion of  the whole $M$ into $\CP^N$(cfr. \cite[Theorem 11]{Cal}).\\

The content of theorems \ref{mainteor1}, \ref{cod3} and of the above mentioned conjecture should be compared with the following results:
\begingroup
\setcounter{tmp}{\value{theor}}
\setcounter{theor}{0} 
\renewcommand\thetheor{\Alph{theor}}
\begin{theor}[D. Hulin \cite{hulinlambda}]
If a compact KE manifold is projectively induced then its Einstein constant is positive.
\end{theor}
\begin{theor}[S. S. Chern \cite{ch}, K. Tsukada \cite{ts}]\label{ct}
Let $(M,g)$ be a complete $n$-dimensional K\"ahler--Einstein manifold ($n\geq 2$). If $(M,g)$ admits a K\"ahler immersion into $\CP^{n+2}$, then $M$ is either totally geodesic or the complex quadric $Q_n$ in $\CP^{n+1}$.
\end{theor}
\endgroup
\setcounter{theor}{\thetmp}

\section{Proof of the main result}
In order to prove theorems \ref{mainteor1} and \ref{cod3}, it is useful to recall the definitions of Calabi's diastasis function and Bochner's coordinates.
Let $(M,g)$ be a K\"ahler manifold with \K\ potential $\Phi$.
If $g$ (and hence $\Phi$) is assumed to be real analytic, by duplicating the variables $z$ and $\bar z$, $\Phi$ can be complex analytically continued to a function $\tilde \Phi$ defined in a neighbourhood $U$ of the diagonal containing $(p,\bar p)\in M\times \bar M$ (here $\bar M$ denotes the manifold conjugated to $M$). Thus one can consider the power expansion of $\Phi$ around the origin  with respect to $z$ and $\bar z$ and write it as
\begin{equation}\label{powexdiastc}
\Phi(z,\bar z)=\sum_{j,k=0}^{\infty}a_{jk}z^{m_j}\bar z^{m_k},
\end{equation}
where we arrange every $n$-tuple of nonnegative integers as a sequence \linebreak$m_j=(m_{j,1},\dots,m_{j,n})$ and order them as follows: $m_0=(0,\dots,0)$ and if $|m_j|=\sum_{\alpha=1}^n m_{j,\alpha}$, $|m_j|\leq |m_{j+1}|$ for all positive integer $j$. Moreover, $z^{m_j}$ denotes  the monomial in $n$ variables $\prod_{\alpha=1}^n z_\alpha^{m_{j,\alpha}}$.\\

A K\"ahler potential is not unique, as it is defined up to an addition with the real part of a holomorphic function.
The \emph{diastasis function} $\dd_0$ for $g$ is nothing but the K\"ahler potential around $p$ such that each matrix $(a_{jk})$ defined according to equation (\ref{powexdiastc}) with respect to  coordinates system centered in $p$, satisfies $a_{j0}=a_{0j}=0$ for every nonnegative integer $j$.

For any real analytic K\"ahler manifold there exists a coordinates system, in a neighbourhood of each point,  such that
\begin{equation}\label{bochnercoordinates}
\dd_0(z)=\sum_{\alpha=1}^n|z_\alpha|^2+\psi_{2,2},
\end{equation}
where $\psi_{2,2}$ is a power series with degree $\geq 2$ in both $z$ and $\bar z$. These coordinates, uniquely determined up to unitary transformation (cfr. \cite{bochner}, \cite{Cal}), are called the \emph{Bochner's coordinates} (cfr. \cite{bochner}, \cite{Cal}, \cite{hulin}, \cite{hulinlambda}, \cite{ruan}, \cite{tian4}).

In order to prove theorem \ref{mainteor1} we need the following two lemmata, the first one dealing with Bochner's coordinates and KE metrics and the second with Bochner's coordinates and projectively induced and rotation invariant metrics.

\begin{lem}[cfr. \cite{note} and \cite{hulinlambda}]\label{bochnerKE}
A K\"ahler manifold $(M,g)$ is Einstein if and only if by choosing Bochner's coordinates on a neighbourhood $U$ of a point $p$ of a K\"ahler manifold $(M,g)$ whose diastasis on $U$ is given by $\dd_0(z)$, it satisfies the Monge--Amp\`ere Equation
\begin{equation}\label{mongebochner}
\det(g_{\alpha\bar\beta})=e^{-\frac{\lambda}{2}\dd_0(z)}.
\end{equation}
\end{lem}
\begin{proof}
Recall that requiring  a K\"ahler metric $g$ on a complex manifold $M$ to be Einstein, i.e. there exists $\lambda\in\R$ such that
\begin{equation}\label{einstein}
\rho=\lambda\omega,
\end{equation}
where $\rho=-i\de\bar\de\log\det(g_{\alpha\bar\beta})$ is the Ricci form associated to $g$, is locally equivalent to require that it satisfies the Monge-Amp\`ere equation
\begin{equation}
\det(g_{\alpha\bar\beta})=e^{-\frac{\lambda}{2}(\Phi+f+\bar f)}.\nonumber
\end{equation}
Therefore, it is easy to check, once once set Bochner's coordinates are set, that the expansion of $\det(g_{\alpha\bar\beta})$ in the $(z,\bar z)$-coordinates around the origin reads $\det(g_{\alpha\bar\beta})=1+h(z,\bar z)$, where $h(z,\bar z)$ is a power series in $z$, $\bar z$ which contains only mixed terms (i.e. of the form $z^j\bar z^k$, $j\neq 0$, $k\neq 0$). Further, also the expansion of $\dd_0(z)$, given in equation (\ref{bochnercoordinates}), contains only mixed terms, forcing $f+\bar f$ to be zero.
\end{proof}
We recall that a holomorphic  immersion $f\!:U\f\CP^N$ is \emph{full} provided $f(U)$ is not contained in any
$\CP^k$ for $k< N$.

\begin{lem}\label{diastpolin}
Let $g$ be a projectively induced K\"ahler metric on a complex  manifold $(M, g)$
and let  $f:V\f \CP^N$ be the holomorphic immersion such that $f^*g_{FS}=g$.
Assume that  $f$ is full and $g$ admits a   diastasis $\dd_0$
on a neighbourhood  $U$ of  a point $p\in M$, which is rotation invariant with respect to the Bochner's coordinates $z_1, \dots, z_n$ around $p$.
Then there exists an open neighbourhood   $W$ of $p$ such that $\dd_0(z)$  can be written on $W$ as
\begin{equation}\label{diastbochpi}
\dd_0(z)=\log\left(1+\sum_{j=1}^n|z_j|^2+\sum_{j=n+1}^N a_j|z^{m_{h_j}}|^2\right),
\end{equation}
where $a_j>0$ and $h_j\neq h_k$ for $j\neq k$.
\end{lem}
\begin{proof}
Up to a unitary transformation of $\CP^N$ and by shrinking $V$ if necessary we can assume $f(p)=[1, 0\dots, 0]$ and $f(V)\subset U_0=\{Z_0\neq 0\}$. Since the affine coordinates on $U_0$ are Bochner's coordinates for the Fubini--Study metric $g_{FS}$ then,   by \cite[Theorem 7 p. 15]{Cal}, $f$
can be written on $W=U\cap V$ as:
$$f:W\f \C^N, z=(z_1, \dots , z_n)\mapsto (z_1, \dots z_n, f_{n+1}(z), \dots , f_N(z)),$$
where
$$f_j(z)=\sum_{k=n+1}^{\infty}\alpha_{jk}z^{m_k}, \ j=n+1, \dots, N.$$
Since the diastasis function is hereditary (see \cite[Prop. 6 p. 4]{Cal} ) and that of  $\CP^N$
around the point $[1, 0\dots, 0]$ is given on $U_0$ by $\Phi(z)=\log(1+\sum_{j=1}^N|z_j|^2)$, where $z_j=\frac{Z_j}{Z_0}$,  one gets
$$\dd_0(z)=\log(1+\sum_{j=1}^n|z_j|^2+\sum_{j=n+1}^N|f_j(z)|^2).$$
The rotation invariance of $\dd_0(z)$ and the fact that $f$ is full  implies that the $f_j$'s are monomial
of $z$ of different degree and formula (\ref{diastbochpi}) follows.
\end{proof}

To simplify the notation, from now on we write $P_{z_\alpha}$ for $\de P/\de z_\alpha$, $P_{\bar z_\beta}$ for $\de P/\de\bar z_\beta$, $P_{z_\alpha \bar z_\beta}$ for $\de^2 P/\de z_\alpha \de\bar z_\beta$ and so on.
\begin{proof}[Proof of Theorem \ref{mainteor1}]
By Lemma \ref{diastpolin} we can assume that there exists $p\in M$ such that the diastasis around $p$ can be written as $\dd_0(z)=\log(P)$ where $P$ is a polynomial in the variables $|z_1|^2,\dots, |z_n|^2$.
By the equality
\begin{equation}
g_{\alpha\bar\beta}=\frac{\de^2 \dd_0(z)}{\de z_\alpha\de\bar z_\beta}=\frac{PP_{ z_\alpha \bar z_\beta}-P_{z_\alpha}P_{\bar z_\beta}}{P^2},\nonumber
\end{equation}
we have
\begin{equation}
\det(g_{\alpha\bar\beta})=\det\left(\frac{PP_{z_\alpha \bar z_\beta}-P_{z_\alpha}P_{\bar z_\beta}}{P^2}\right)=\frac{1}{P^{2n}}\det\left(PP_{z_\alpha \bar z_\beta}-P_{z_\alpha}P_{\bar z_\beta}\right).\nonumber
\end{equation}
Given a polynomial $Q$ in the variables $z_1,\dots,z_n,\bar z_1,\dots,\bar z_n$ we denote by $\deg Q$ the total degree of $Q$ with respect to all the variables $z_1,\dots,z_n,\bar z_1,\dots,\bar z_n$. Then
\begin{equation}
\deg \det\left(PP_{z_\alpha \bar z_\beta}-P_{z_\alpha}P_{\bar z_\beta}\right) \leq 2n\deg P - 2n. \nonumber
\end{equation}
On the other hand from Monge--Amp\`ere Equation (\ref{mongebochner}) we get
\[
\deg \det\left[\left(PP_{z_\alpha \bar z_\beta}-P_{z_\alpha}P_{\bar z_\beta}\right)\right] - 2n \deg  P = -\frac{\lambda}{2} \deg P , \]
which forces $\frac{\lambda}{2} \geq\frac{2n}{deg  P} > 0$. Thus, if $M$ is complete, by Bonnet's Theorem, $M$ is also compact.  Then $M$ is  simply connected by a well-known theorem of Kobayashi \cite{koricci}. 
The final upper bound $ \lambda \leq 2(n+1)$ follows from the following D. Hulin's result.
\end{proof}

\begin{lem}[D. Hulin \cite{hulin}]\label{hulin}
Let $(V,h)$ be a KE manifold which admits a K\"ahler immersion into $\CP^N$. Then it can be extended to a complete n-dimensional KE manifold $(M,g)$ and the Einstein constant is a rational number. Further, let this immersion be full and let the Einstein constant $\lambda=2\frac{p}{q}$ be positive, where $p/q$ is irreducible, then $p\leq n+1$ and if $p= n+1$ (resp. $n=p=2$) then $(M,g)=(\CP^n,qg_{FS})$ (rep. $(M,g)=(\CP^1\times\CP^1,g_{FS}\oplus g_{FS})$).
\end{lem}

In order to prove Theorem \ref{cod3} we also need:

\begin{lem}\label{gencod}
Let $(M,g)$ be a complete $n$-dimensional rotation invariant KE manifold. If $n>2k$, $(M,g)$ admits a (local) K\"ahler immersion into $\CP^{n+k}$ if and only if $(M,g)=(\CP^n,g_{FS})$.
\end{lem}

\begin{proof}
By Lemma \ref{diastpolin} there exist a point $p\in M$ and local coordinates $z_1,\dots,z_n$ around it such that the diastasis function $\dd_0(z)$ for $g$ centered at $p$ can be written as
\begin{equation}
\begin{split}
\dd_0(z)=\log(1&+\sum_{j=1}^n|z_j|^2+\sum_{j=1}^na_j|z_j|^4+\\
&+\sum_{1\leq j<k\leq n}b_{jk}|z_j|^2|z_k|^2+\psi_{3,3}),\nonumber
\end{split}\nonumber
\end{equation}
where $\psi_{3,3}$ is a (rotation invariant) polynomial of degree not less than three both in $z$ and in $\bar z$.
For $h=1,\dots, n$, deriving with respect to $z_{h}$ and $\bar z_h$ both sides of the Monge--Amp\`ere Equation (\ref{mongebochner}) and evaluating at $0$, i.e. by considering the $n$ equalities
\begin{equation}
\frac{\de^2}{\de z_h\de\bar z_h}\left(\det(\frac{\de^2 \dd_0(z)}{\de z_j\de\bar z_k})\right)\big|_0=\frac{\de^2}{\de z_h\de\bar z_h}\left(e^{-\frac{\lambda}{2}\dd_0(z)}\right)\big|_0,\ (h=1,\dots, n),\nonumber
\end{equation}
we get $n$ equations of the form
\begin{equation}\label{deinst}
4a_h+\sum_{\substack{k=1\\ k\neq h}}^nb_{hk}-(n+1)=-\frac{\lambda}{2}, \quad (h=1,\dots, n),
\end{equation}
in which $b_{ij}=b_{ji}$. Thus, if the codimension of $M$ into $\CP^N$ is $k<n/2$ at least one of these equations is of the form $\lambda=2(n+1)$, because the polynomial $P$ consists of only $n+k+1$ monomials and then at most $k$ of the variables $\{a_i, b_{ij}\}_{\substack{1\leq i\leq n\\ i<j\leq n}}$ can be different from $0$. Therefore the thesis follows by Lemma \ref{hulin}.
\end{proof}

We are now in the position of proving Theorem \ref{cod3}.

\begin{proof}[Proof of Theorem \ref{cod3}]
By virtue of Lemma \ref{hulin}, we can consider directly the case where $(M,g)$ is a complete KE projectively induced and rotation invariant manifold.\\
By Lemma \ref{diastpolin} there exist a point $p\in M$ and local coordinates $z_1,\dots, z_n$ such that the diastasis function $\dd_0(z)$ around $p$ can be written as
\begin{equation}
\begin{split}
\dd_0(z)=\log\left(1+\sum_{j=1}^n|z_j|^2 \right.+&\sum_{j=1}^na_j|z_j|^4+ \sum_{1\leq j<k\leq n}b_{jk}|z_j|^2|z_k|^2+\\
+&\left.\sum_{j=1}^nc_j|z_j|^6+\sum_{j,k=1 (j\neq k)}^nd_{jk}|z_j|^2|z_k|^4+\psi_{4,4}\right),\nonumber
\end{split}
\end{equation}
where $\psi_{4,4}$ is a rotation invariant polynomial of degree not less than four both in $z$ and in $\bar z$.
If $k=0$, we have $\dd_0(z)=\log(1+\sum_{j=1}^n|z_j|^2)$ and $(M,g)=(\CP^n,g_{FS})$. The statement holds by Theorem \ref{ct} for those case where the codimension is equal to $1$ or $2$  . Thus let $k=3$. By Lemma \ref{gencod}, the statement is true for $n>6$. Here we need to analyze the cases $n\leq 6$ separately.

Consider the case $n=2$. As in the proof of Lemma \ref{gencod}, for $h=1,2$, by deriving with respect to $z_h$, $\bar z_h$ both sides of the Monge--Amp\`ere Equation (\ref{mongebochner}), and evaluating at $0$ we get the $2$ equations
\begin{equation}
\begin{split}
 4 a_1+b_{12}& =3-\frac{\lambda}{2},  \\
 4 a_2+b_{12}& =3-\frac{\lambda}{2},\nonumber
\end{split}
\end{equation}
from which follows immediately $a_1=a_2$ and $b_{12} =3-\frac{\lambda}{2}-4a_1$.
If $a_1=a_2=b_{12}=0$, we get $\lambda=6$, and by Lemma \ref{hulin} $(M,g)=(\CP^2,g_{FS})$. If $a_1\neq0$, $b_{12}\neq 0$, since the codimension is $3$, all the other coefficients must vanish, thus by deriving both sides of the Monge--Amp\`ere Equation (\ref{mongebochner}) by $z_1$, $\bar z_1$, $z_2$, $\bar z_2$, and twice by $z_1$ and $\bar z_1$, evaluating at zero, we get a system, whose unique acceptable solution is
$b_{12}=1/2$, $a_1=1/4$ and $\lambda=3$, that is $(M,g)=(\CP^2,2g_{FS})$. If $a_1\neq 0$ but $b_{12}=0$, by deriving again both sides of the Monge--Amp\`ere Equation (\ref{mongebochner}) by $z_1$, $\bar z_1$, $z_2$, $\bar z_2$ and evaluating at zero and at $\lambda=2(3-4a_1)$ we have $4a_1+4d_{12}+4d_{21}=0$, impossible, since $a_1\neq 0$ and all the coefficients must be non negative. It remains to consider the case $b_{12}\neq 0$, $a_1=a_2=0$. By deriving the Monge--Amp\`ere Equation twice by $z_j$ and twice by $\bar z_j$ (for $j=1,2$), evaluating at zero and at $a_1=a_2=0$, $b_{12}=3-\frac{1}{2}\lambda$, we get $d_{12}=9c_1$ and $d_{21}=9c_2$. Since the codimension is $3$, only two of them can be different from zero. If they are all zero, we get $\lambda=6$ and by Lemma \ref{hulin} $(M,g)=(\CP^2,g_{FS})$ or $\lambda=4$ and again by Lemma \ref{hulin} $(M,g)=(\CP^1\times\CP^1,g_{FS}\oplus g_{FS})$. Let us suppose that two of them are different from zero, for example $d_{12}\neq 0$. Then all the terms of higher order vanish, and taking the third order derivative we get again $\lambda=6$ or $\lambda=4$.

The case $n=3$ is very similar to that one. The system given in the proof of Lemma \ref{gencod} reads
\begin{equation}
\begin{split}
 4 a_1+b_{12}+b_{13} & =4-\frac{\lambda}{2},  \\
  4 a_2+b_{12}+b_{23}& =4-\frac{\lambda}{2},   \\
4  a_3+b_{13}+b_{23} & =4-\frac{\lambda}{2}. \\
\end{split}\nonumber
\end{equation}
It is easy to see that only three cases do not reduce immediately to $(M,g)=(\CP^3,g_{FS})$, that is $a_1=a_2=a_3\neq 0$, or $a_1=b_{23}\neq 0$ (and all the symmetric to them), or $b_{12}=b_{23}=b_{13}\neq 0$. By taking the second order derivative of the Monge--Amp\`ere Equation  and evaluating at zero, it follows that these cases are incompatible.

If $n=4$, it easy to see from the system of linear equation
\begin{equation}
\begin{split}
 4 a_1+b_{12}+b_{13}+b_{14} & =5-\frac{\lambda}{2},\\
  4 a_2+b_{12}+b_{23}+b_{24}& =5-\frac{\lambda}{2},\\
4  a_3+b_{13}+b_{23}+b_{34} & =5-\frac{\lambda}{2},\\
4 a_4+b_{14}+b_{24}+b_{34}  & =5-\frac{\lambda}{2},
\end{split}\nonumber
\end{equation}
that, up to symmetries, only the following cases may occur:
all the coefficients are equal to zero, $a_1=a_2=b_{34}\neq 0$ or $b_{12}=b_{34}\neq 0$. In the third case , without loss of generality, we suppose that $d_{23}=d_{32}=0$. Therefore if the second or third case holds, by deriving both sides of the Monge--Amp\`ere Equation with respect to $z_2$, $z_3$, $\bar z_2$, $\bar z_3$, evaluating at zero and considering the relation above, we get $b_{34}=0$, and conclusion follows.

The cases $n=5$ and $n=6$ are very similar to that one. For $n=5$, by the system of linear equation either the coefficients of the system are zero, or up to symmetries $4a_1=b_{23}=b_{45}\neq 0$. Deriving with respect to $z_2$, $z_4$, $\bar z_2$, $\bar z_4$ the Monge--Amp\`ere Equation and evaluating at zero, one gets $b_{23}=0$. For $n=6$, from the system of linear equation one gets that either the coefficients are all zero or $b_{12}=b_{34}=b_{56}\neq 0$. By deriving with respect to $z_2$, $z_4$, $\bar z_2$, $\bar z_4$ and evaluating at zero, one gets $b_{34}=0$, and we are done. 
\end{proof}

\vspace{0.2cm}

\section*{Acknowledgements}
The author wishes to thank prof. Andrea Loi for many helpful discussions and for his advice concerning various aspects of this work.

\vspace{0.1cm}

\small{}

\end{document}